\documentclass[a4paper,10pt]{amsart}
\usepackage{amsmath,amssymb,amsthm,enumerate,epsfig}
\newcommand{\dis}{\displaystyle}
\newcommand{\codim}{\operatorname{codim}}

\def\N{{\mathbb N}}

\def\A{{\mathcal A}}
\def\G{{\mathcal G}}
\def\E{{\mathcal E}}

\def\Z{{\mathbb Z}}

\def\a{{\mathbf a}}
\def\e{{\mathbf e}}
\def\b{{\mathbf b}}

\theoremstyle{plain}
\newtheorem{thm}{Theorem}[section]

\newtheorem{lemma}[thm]{Lemma}
\newtheorem{prop}[thm]{Proposition}

\newtheorem{rem}[thm]{Remark}

\theoremstyle{definition}

\theoremstyle{remark}

\begin{document}

\title[Toric Ideals]{Toric Ideals of Simple Surface Singularities}

%    author one information

\author{G\"ulay Kaya}
\address{Department of Mathematics, Galatasaray~University, \.{I}stanbul, Turkey}
\curraddr{} \email{gukaya@gsu.edu.tr}
\thanks{}

%    author two information
\author{P\i nar Mete}
\address{Department of Mathematics, Bal\i kesir University, Bal\i kesir, 10145 Turkey}
\curraddr{} \email{pinarm@balikesir.edu.tr}
\thanks{}

%    author three information
\author[Mesut \c{S}ah\.{i}n]{Mesut \c{S}ah\.{i}n}
\address{Department of Mathematics, \c{C}ank\i r\i ~Karatek\.{i}n University, \c{C}ank\i r\i, 18100 Turkey}
\curraddr{} \email{mesutsahin@gmail.com}



\date{\today}

%\dedicatory{}
\begin{abstract}
In this paper, we study a class of toric ideals obtained by using some geometric data of ADE trees which are the minimal resolution graphs of rational surface singularities. We compute explicit Gr\"obner bases for these toric ideals that are also minimal generating sets consisting of large number of binomials of degree $\leq 4$. In particular, they give rise to squarefree initial ideals as well. 
\end{abstract}

\maketitle

\section{Introduction}
Algebraic varieties having squarefree initial ideals are of
special interest. Many authors have presented squarefree initial
ideals arising from different contexts, see for instance
\cite{bruns,haase,pfaffians,kit,oh1}. Normal toric ideals are known to
have at least one squarefree term in each minimal binomial
generator by \cite[Proposition 4.1]{simis} and \cite[Lemma
6.1]{oh2}. They have Cohen-Macaulay initial ideals when their configurations are 
$\Delta$-normal, see \cite{OShea-Thomas}. These suggest that they have (at least 
simplicial ones) squarefree initial ideals with respect to a term order. The
challenge lies in the choice of a correct term order. Motivated by fundamental questions
in combinatorial commutative algebra and its applications to statistics and optimization, recently,
with the aid of Gale diagrams, Dueck et al. \cite{DHS} have
succeeded to show the existence of a term order with respect to which normal toric ideals of codimension 
$2$ have squarefree initial ideals. They have also proven that the Gr\"{o}bner bases giving rise to these initial
ideals constitute minimal generating sets for the toric ideals.

The aim of the present paper is to extend the discussion to certain examples of normal 
toric ideals of higher codimension. As a case study, we concentrate on certain toric ideals
of higher codimension arising from singularity theory that are promising because of the speciality
of the corresponding singularities. These are the simplicial normal toric
ideals corresponding to the simple or $ADE$ surface singularities.
In section 3, we prove that toric ideals of $DE$ type singularities have squarefree initial ideals.
Our methods are computational and use the configurations given in
\cite{meral-selma}. The reduced Gr\"{o}bner bases we obtain are
also shown to be minimal generating sets containing a large number
of binomials of degree at most $4$, see section \ref{sec04}. In
the last section, we speculate on initial ideals of $A_n$-type
trees whose configurations seem impossible to give a closed form.

\section{Preliminaries}

\subsection{Gr\"{o}bner basis.} Let $\A=\{\mathbf{a_1},\dots,\mathbf{a_N}\}$ be a configuration in
$\mathbb{Z}^n$  and $K[\A]:=K[\{x_\mathbf{a} \,| \,\mathbf{a} \in
\A\}]$ denote the polynomial ring in variables $x_\mathbf{a}$ with
$\mathbf{a}\in \A$ over the field $K$. Consider the affine
semigroup $\N\A =
\{\lambda_1\mathbf{a_1}+\cdots+\lambda_N\mathbf{a_N}  \; : \;
\lambda_i \in \N \}$ and let
$K[\N\A]:=K[\{\mathbf{u^{a}}\,|\,\mathbf{a}\in \A\}]$ be the
associated semigroup ring. The \textit{toric ideal} $I_{\A}$ of
$\A$ is the kernel of the following $K$-algebra epimorphism:
$$\pi:K[\A]\rightarrow K[\N\A],
\quad \pi(x_\mathbf{a}):=\mathbf{u^{a}}=u_1^{a_{1}} \cdots
u_n^{a_{n}}.$$ It is known that $I_{\A}$ is a prime ideal
generated by binomials $x_\mathbf{a}-x_\mathbf{b}$ with
$\pi(x_\mathbf{a})=\pi(x_\mathbf{b})$ \cite{sturmfels}. The zero
set of $I_{\A}$ is called the toric variety $V_{\A}$ of $\A$.

The {\it initial monomial}, $in(f)$, of a polynomial $f\in I_{\A}
\setminus \{0\}$ is the greatest monomial of $f$ with respect to a
term order on the monomials of $K[\A]$. The \emph{initial ideal},
$in(I_{\A})$, of $I_{\A}$ is a \emph{monomial} ideal generated by
all initial monomials of polynomials in $I_{\A}$. A finite subset
$\G \subset I_{\A}$ is called a Gr\"obner basis of $I_{\A}$ if $in
(I_{\A})=in (\G)$, where $in (\G)$ is the monomial ideal generated
by initial monomials of polynomials in $\G$. The following is the
key in proving our main results.

\begin{lemma}\cite[Lemma 1.1]{ahh} \label{lem} With the preceding notation, let $M$ and $M'$ be monomials in $K[\A]$.
The finite set $\G$ is a Gr\"obner basis of $I_{\A}$ if
and only if $\pi(M)\neq \pi(M')$ for all $M \notin in(\G)$ and $M'
\notin in(\G)$ with $M \neq M'$.
\end{lemma}

%Since \textit{(graded) revlex}  monomial order will specifically be used to obtain
%the reduced Gr\"obner basis of toric ideals, we give its definition.

%\begin{definition}[(Graded) Reverse Lex Order]
%Let $\alpha, \beta \in \mathbb{Z}^n_{\geq 0}$. $\alpha >_{revlex} \beta$,
%iff

%$$\mid \alpha\mid = \sum_{i=1}^{n} \alpha_{i}\, > \, \mid \beta \mid = \sum_{i=1}^{n} \beta_{i},
%\;or\;  \mid \alpha\mid\, =\, \mid \beta\mid$$

%\vspace{.2cm}
%\noindent and $\alpha_{i} < \beta_{i}$ for the last index \emph{i} with $\;\alpha_{i}\, \neq\, \beta_{i} \}$.
%\end{definition}

%\begin{notation}
%Throughout this paper, we assume that the first
%term of a binomial is its initial monomial for a fixed term order.
%\end{notation}

\subsection{$ADE$-trees} Here, we briefly review basics of $ADE$-trees, see \cite{bourbaki,Artin,zariski,DuVal,Grauert} for more details. Let $\Gamma $ be a weighted graph without loops, with vertices $C_1,\dots
,C_n$ and with weight $w_i\geq 2$ at each vertex $C_i$. The
incidence matrix ${\mathcal M}(\Gamma )=[c_{ij}]$, associated with
$\Gamma $ is a symmetric matrix and defined in the following way:
$c_{ii}=-w_i$ and $c_{ij}$ is the number of edges linking the
vertices $C_i$ and $C_j$ whenever $i\not =j$. On the free abelian group $\mathcal{L}$
generated by  the vertices $C_i$ of $\Gamma$, ${\mathcal M}({\Gamma })$ defines
a symmetric bilinear form $(Y\cdot Z)$ for a pair $(Y,Z)$ of
elements in $\mathcal L$ via $(C_i\cdot C_j):=c_{ij}$. 
The elements $C=\sum_{i=1}^{n}m_iC_i$ of $\mathcal L$ will be called
{\it cycles} of the graph $\Gamma$ where $m_i \in  \Z$. A {\it positive cycle} is a
non-zero cycle with non-negative coefficients.

\vspace{.2cm} If $w_i=2$ for all $i$ and $\displaystyle C\cdot
C\leq -2$ for any cycle then $\Gamma $ is of type $A_{n}$, $D_n$,
$E_6$, $E_7$ and $E_8$. It is well known that these are the Dynkin
diagrams obtained as the minimal resolution graphs of the rational
singularities of complex surfaces. The semigroup of Lipman is the
set
$$\E^{+}(\Gamma ):=\{ C \in \mathcal L \mid  (C \cdot C_i)\leq 0 \quad\mbox{for}\quad 1\leq i \leq n\}.$$
which is not empty since $M(\Gamma )$ is negative definite in this case. By \cite{lipman}, each  element of this set corresponds to a function in the maximal ideal of the local ring of the singularity on the surface having $\Gamma $ as the minimal resolution graph.

In \cite{meral} and \cite{meral-selma}, the authors have studied
the structure of this semigroup and provided an algorithm to find
a generating set over $\Z $ by associating an affine toric variety $V_{\A}$, c.f. also \cite{mesut}.
This toric variety corresponds to the configuration $\A$ of the
smallest $n$-tuples $(d_1,\ldots ,d_n)\in \N^n$ such that
$(C\cdot C_i)=-d_i$ for $C\in \E^{+}(\Gamma )$.  The interested
reader can see \cite{meral-selma} for the details.

\section{Squarefree initial ideals} \label{sec03} In this section, we obtain reduced
Gr\"obner bases for toric ideals of affine toric varieties
corresponding to $DE$-type singularities. Throughout the section, we assume that the first term of a 
binomial is its initial monomial for a fixed term order. In order to find the
set $\A$ which determines the parametrization of the toric variety
$V_{\A}$, we use Proposition $3.9$ and $3.12$ in
\cite{meral-selma}.
\subsection{$D_n$-type singularities} We have $n \geq 4$.
Since toric ideals behave in a different manner when $n$ is
even and odd, we discuss two cases separately.

\smallskip

\noindent {\bf When $n=2m$:} Let $J=\{3,5,\dots,n-1\}$ and
$J^c=\{2,4,\dots,n-2\}$. Consider the subset
$$\dis D_{2m}:=\{2\e_i, \e_j, 2\e_{1},2\e_{n}, \e_k+\e_{\ell},\e_i+\e_{1}+\e_n
\,|\, i,k, \ell \in J, \, j \in J^c \,\, \mbox{and} \,\, k<\ell \},$$ where
$\{\e_1,\dots,\e_n\}$ is the canonical basis of $\Z^n$. Then we introduce one variable for each element in the set $D_{2m}$ and define the polynomial ring $K[D_{2m}]$ to be the $K$-algebra generated by the
set of these variebles
$$\dis \{x_1,\dots,x_n, x_{j,k},y_{i}\;|
\,\, \mbox{where}\,\, i,j,k \in J \,\, \mbox{and} \,\, j<k \}.$$
Similarly we define the semigroup ring $K[\N D_{2m}]$ to be the $K$-algebra generated
by $$\dis \{u_i^2, u_j, u_{1}^2, u_{n}^2, u_ku_{\ell},u_iu_{1}u_n
\,|\, i,k, \ell \in J, \, j \in J^c \,\, \mbox{and} \,\, k<\ell
\}.$$

\noindent The toric ideal $I_{D_{2m}}$ is thus the kernel of $\pi:K[D_{2m}]\rightarrow K[\N D_{2m}]$ which is defined as:
$$\pi(x_{i})=u_i^2, \; \pi(x_{j})=u_j, \;\pi(x_{1})=u_{1}^2, \; \pi(x_{n})=u_{n}^2,
\;\pi(x_{k,\ell})=u_ku_{\ell},\\$$
$$\pi(y_i)=u_iu_{1}u_{n}$$
for all $i,k, \ell \in J, \, j \in J^c$ with $k<\ell$.

We next define the ordering $\succ^{even}$ to be the reverse lexicographic ordering imposed by:
$$x_{1}\succ \dots \succ x_{n-1} \succ x_n \succ x_{j_1,j_2}\succ x_{j_3,j_4}\succ y_{k_1}\succ y_{k_2}$$
where $j_1,j_2,j_3,j_4,k_1,k_2\in J$ with $j_2
< j_4$ or $j_2=j_4, \;j_1<j_3$; and $k_1<k_2$. 

Then a squarefree initial ideal for $I_{D_{2m}}$  is given by the
following theorem, since the first monomial of a binomial is its
initial term.
\begin{thm}\label{evenaffine}
The following set $\mathcal{G}_{D_{2m}}$
$$\begin{array}{lllll}
&x_{i,k}x_{j,\ell}-x_{i,j}x_{k,\ell} \quad &x_{i,\ell}x_{j,k}-x_{i,j}x_{k,\ell} \quad &i<j<k<\ell \\
&x_{i,j}x_{i,k}-x_ix_{j,k} \quad &x_jx_{i,k}-x_{i,j}x_{j,k} \quad &i<j<k \\
&x_kx_{i,j}-x_{i,k}x_{j,k} \quad & x_{j,k}y_i-x_{i,j}y_k \quad & i<j<k \\
&x_{i,k}y_j-x_{i,j}y_k \quad & \quad & i<j<k \\
&x_ix_j-x_{i,j}^2 \quad &x_jy_i-x_{i,j}y_j & i<j\\
&x_{i,j}y_i-x_iy_j \quad &x_{i,j}x_{1}x_{n}-y_iy_j \quad & i<j\\
&x_ix_{1}x_n-y_i^2 \quad & \quad & i \in J
\end{array}$$ is a Gr\"obner basis of $I_{D_{2m}}$ with respect to the ordering $\succ^{even}$ defined above.
\end{thm}

\begin{proof} Let $M$ and $M'$ be two monomials in $K[D_{2m}]$ with $M \notin in(\mathcal{G}_{D_{2m}})$ and $M' \notin
in(\mathcal{G}_{D_{2m}})$, where $in(\G_{D_{2m}})$ is the
monomial ideal generated by initial terms of binomials in
$\G_{D_{2m}}$. Since $x_ix_j \in in(\G_{D_{2m}})$, we may assume
that
$$M=x_{a}^px_{1}^{\alpha_{1}}x_{n}^{\alpha_n}x_{b_1,c_1}\cdots
x_{b_q,c_q} y_{d_1}\cdots y_{d_r} \quad \mbox{and} \quad$$
$$M'=x_{a'}^{p'}x_{1}^{\alpha'_{1}}x_{n}^{\alpha'_n}x_{b'_1,c'_1}\cdots
x_{b'_{q'},c'_{q'}} y_{d'_1} \cdots y_{d'_{r'}},\quad \mbox{where}
\quad$$
$$x_{a}\succ x_{b_1,c_1} \succ \cdots \succ x_{b_q,c_q}\succ
y_{d_1} \succ \cdots \succ y_{d_r},$$
$$x_{a'} \succ x_{b'_1,c'_1} \succ \cdots \succ
x_{b'_{q'},c'_{q'}}\succ y_{d'_1} \succ \cdots \succ
y_{d'_{r'}}.$$

First, we observe that the ordering above implies that $c_1 \leq
\dots \leq c_q$, $c'_1 \leq \dots \leq c'_{q'}$ and $d_1 \leq
\dots \leq d_r$, $d'_1 \leq \dots \leq d'_{r'}$. Moreover, we have
$b_1<c_1 \leq b_2<c_2 \leq \dots \leq b_q<c_q$  and $b'_1<c'_1
\leq b'_2<c'_2 \leq \dots \leq b'_{q'}<c'_{q'}$, since
$x_{i,k}x_{j,\ell}, x_{i,\ell}x_{j,k}, x_{i,j}x_{i,k} \in
in(\G_{D_{2m}})$.

 The images of $M$ and $M'$ are found easily as
$$\pi(M)=u_{a}^{2p}u_{1}^{2\alpha_{1}+r}u_{n}^{2\alpha_n+r} u_{b_1}u_{c_1}\cdots
u_{b_q}u_{c_q} u_{d_1}\cdots u_{d_r}$$
$$\pi(M')=u_{a'}^{2p'}u_{1}^{2\alpha'_{1}+r'}u_{n}^{2\alpha'_n+r'}
u_{b'_1}u_{c'_1}\cdots u_{b'_{q'}}u_{c'_{q'}} u_{d'_1}\cdots
u_{d'_{r'}}.$$ In what follows we will prove that $\pi(M)=\pi(M')
\Rightarrow M=M'$, by the virtue of Lemma \ref{lem}. It follows
from $\pi(M)=\pi(M')$ that we have the following identities
\begin{eqnarray}
 \label{e9} 2\alpha_{1}+r &=& 2\alpha'_{1}+r' \\
  \label{e10} 2\alpha_{n}+r &=& 2\alpha'_{n}+r' \\
  \label{e11} 2p+2q+r &=& 2p'+2q'+r' \\
   \label{e12} \alpha_{1}-\alpha_{n} &=& \alpha'_{1}-\alpha'_{n} \quad \quad \quad (\mbox{follows directly from} \: (\ref{e9}) \: \mbox{and} \: (\ref{e10})).
\end{eqnarray}

To accomplish our goal $M=M'$, we will assume now that $M
\neq M'$ to obtain a contradiction in all possible cases
considered below. Since $M\neq M'$, we may suppose further that
they have no variable in common without loss of generality. This is because $in(\G_{D_{2m}})$ is an ideal and $M, M' \notin in(\G_{D_{2m}})$ implies that the new monomials obtained by dividing $M$ and $M'$ by their greatest common divisor will also lie outside of $in(\G_{D_{2m}})$.

If $\alpha_{1}>0$ and $\alpha_n >0$ then
$\alpha'_{1}=\alpha'_{n}=0$, as $M$ and $M'$ have no common
variable. Since $x_{i,j}x_{1}x_n, \,x_{k}x_{1}x_n \in
in(\G_{D_{2m}})$, we have $p=q=0$. This implies that
$r=2p'+2q'+r'$ by $(\ref{e11})$ and thus $2p'+2q'+2\alpha_{1}=0$
by $(\ref{e9})$, a contradiction.

If $\alpha_{1}>0$ and $\alpha_n =0$ then $\alpha'_{1}=0$ which
implies together with $(\ref{e12})$ that $\alpha_{1}=-\alpha'_{n}
\leq 0$, contradiction. The case $\alpha_{1}=0$ and $\alpha_n
>0$ is done similarly. So, we have only the case where
$\alpha_{1}=0$ and $\alpha_n =0$. A similar argument shows that
$\alpha'_{1}=\alpha'_n=0$. In this case $r=r'$ by $(\ref{e9})$.

\textbf{Case I:} Assume $r=r'>0$. Since $x_jy_i \in
in(\G_{D_{2m}})$, for all $i<j$, it follows that $a \leq d_r$.
Again by $x_{i,j}y_i, x_{j,k}y_i, x_{i,k}y_j \in
in(\G_{D_{2m}})$, for all $i<j<k$, we have $(b_q<)c_q \leq d_r$
and $(b'_{q'}<)c'_{q'} \leq d'_{r'}$. Hence, $d_r$ (resp. $d'_r$)
is the biggest index appearing in $\pi(M)$ (resp. $\pi(M')$).
Since $\pi(M)=\pi(M')$, it follows that $d_r=d'_r$. But this
implies that $y_{d_r}$ is a variable appearing in both $M$ and
$M'$, contradiction.

\textbf{Case II:} Assume $r=r'=0$. If $q=0$ then $\pi(M)=\pi(M')$
implies that $u_{a}^{2p}= u_{a'}^{2p'}u_{b'_1}u_{c'_1}\cdots
u_{b'_{q'}}u_{c'_{q'}}$, which is possible only if $q'=0$ as
$b'_{q'}<c'_{q'}$. But in this case $a=a'$ and $x_{a}$ is a
common variable of $M$ and $M'$, a contradiction. Thus $q>0$ and
$q'>0$.

Since $x_j x_{i,k}, x_k x_{i,j} \in in(\G_{D_{2m}})$, we have
$a \leq c_q$ and $a' \leq c'_{q'}$.  Since $b_q < c_q$ and
$b'_{q'} < c'_{q'}$, we observe that $c_q$ (resp. $c'_{q'}$) is
the biggest index appearing in $\pi(M)$ (resp. $\pi(M')$) which
yields together with $\pi(M)=\pi(M')$ that $c_q = c'_{q'}$. In
this case $u_{b_q}$ and $u_{b'_{q'}}$ appear in $\pi(M)=\pi(M')$.
Clearly $b_q>b'_{q'}$ or $b_q<b'_{q'}$, as otherwise $M$ and $M'$
would have a common variable $x_{b_{q},c_{q}}$. If
$b_q>b'_{q'}(>\dots>b'_1)$ then $b_q=a'$ as $u_{b_q}$ appears in
$\pi(M')$. This forces that $b'_{q'}<b_q=a'<c_q=c'_{q'}$ which
is impossible, since $x_jx_{i,k} \in in(\G_{D_{2m}})$. The other
case $b_q<b'_{q'}$ is impossible by a similar argument.
\end{proof}

\begin{rem}
Note that we have
$$|\mathcal{G}_{D_{2m}}|=2\binom{m-1}{4}+5 \binom{m-1}{3}+4\binom{m-1}{2}+\binom{m-1}{1}.$$
$\dim V_{D_{2m}}=2m$, $\codim V_{D_{2m}}=m-1+\binom{m-1}{2}$.
\end{rem}

\smallskip

\noindent {\bf When $n=2m+1$ :} Let $J=\{2,4,\dots,n-1\}$ and
$J^c=\{3,5,\dots,n-2\}$. Consider the subset $D_{2m+1}$ defined by
$$\dis \{2\e_i,\e_j,4\e_{1},4\e_{n}, \e_k+\e_{\ell},\e_{1}+\e_n,
\e_i+2\e_{1},\e_i+2\e_{n}, \e_i+3\e_{1}+\e_n, \e_i+\e_{1}+3\e_n$$ 
$$\,|\, i,k,\ell \in J, \,j\in J^c \,\, \mbox{and} \,\, k<\ell \},$$
where $\{\e_1,\dots,\e_n\}$ is the canonical basis of $\Z^n$. As before we introduce one variable for each member of $D_{2m+1}$ and define the polynomial ring $K[D_{2m+1}]$ to be the $K$-algebra generated by
the set
$$\dis \{x_1,\dots,x_n, x_{j,k},x_{1,n},x_{i,1},x_{i,n},y_{i,1},y_{i,n}\;|
\,\, \mbox{where}\,\, i,j,k \in J \,\, \mbox{and} \,\, j<k \}$$
and the semigroup ring $K[\N D_{2m+1}]$ to be the $K$-algebra
generated by $$\dis \{u_i^2, u_j, u_{1}^4, u_{n}^4,
u_ku_{\ell},u_{1}u_n,
u_iu_{1}^2,u_iu_{n}^2,u_iu_{1}^3u_{n},u_iu_{1}u_{n}^3 \,|\,
i,k,\ell \in J, \,j\in J^c \,\, \mbox{and} \,\, k<\ell \}.$$

\noindent The toric ideal $I_{D_{2m+1}}$ is thus the kernel of $\pi:K[D_{2m+1}]\rightarrow K[\N D_{2m+1}]$ which is defined as follows:
$$\pi(x_{i})=u_i^2, \; \pi(x_{j})=u_j, \;\pi(x_{1})=u_{1}^4, \; \pi(x_{n})=u_{n}^4,
\;\pi(x_{k,\ell})=u_ku_{\ell}, \; \pi(x_{1,n})=u_{1}u_{n}\\$$
$$\pi(x_{i,1})=u_iu_{1}^2, \; \pi(x_{i,n})=u_{i}u_{n}^2, \;
\pi(y_{i,1})=u_iu_{1}^3u_{n}, \; \pi(y_{i,n})=u_iu_{1}u_{n}^3$$
for all $i,k,\ell \in J, \,j\in J^c$ with $k<\ell$.

Finally, we define the ordering $\succ^{odd}$ to be the reverse lexicographic
ordering imposed by:
$$y_{i_1,1}\succ y_{i_2,1} \succ y_{i_1,n}\succ y_{i_2,n} \succ x_{1} \succ \dots \succ x_n \succ$$ $$\succ x_{j_1,j_2}\succ x_{j_3,j_4} \succ x_{k_1,1}\succ x_{k_2,1}\succ x_{\ell_1,n}\succ x_{\ell_2,n}\succ x_{1,n}$$
where $j_1,j_2,j_3,j_4,k_1,k_2,\ell_1,\ell_2 \in J$ with $j_2 <
j_4$ or $j_2=j_4, \;j_1<j_3$ and $k_1<k_2$ and $\ell_1<\ell_2$.

Then a squarefree initial ideal for $I_{D_{2m+1}}$ is given by the
following theorem as the first monomials are the initial terms with respect to the ordering $\succ^{odd}$.

\begin{thm} \label{oddaffine}
The following set $\mathcal{G}_{D_{2m+1}}$
$$\begin{array}{lllll}
&x_{i,k}x_{j,\ell}-x_{i,j}x_{k,\ell} \qquad  &x_{i,\ell}x_{j,k}-x_{i,j}x_{k,\ell}  \qquad &i<j<k<\ell \in J\\
&x_{j,k}x_{i,n-1}-x_{i,j}x_{k,n-1}\qquad &x_{i,k}x_{j,n-1}-x_{i,j}x_{k,n-1}\qquad  &i<j<k\in J\\
&x_{j,k}x_{i,n}-x_{i,j}x_{k,n}\qquad &x_{i,k}x_{j,n}-x_{i,j}x_{k,n}\qquad & i<j<k\in J\\
&x_{j}x_{i,k}-x_{i,j}x_{j,k}\qquad &x_{i,j}x_{i,k}-x_{i}x_{j,k}\qquad & i<j<k\in J\\
&x_{k}x_{i,j}-x_{i,k}x_{j,k}\qquad & \quad & i<j<k\in J\\
&x_ix_j-x_{i,j}^2\qquad &x_{i,j}x_{1}-x_{i,n-1}x_{j,n-1}\qquad &i<j\in J\\
&x_{i,j}x_{n}-x_{i,n}x_{j,n}\qquad &x_{i,j}x_{i,n-1}-x_ix_{j,n-1}\qquad & i<j\in J\\
&x_{i,j}x_{i,n}-x_ix_{j,n}\qquad &x_{j}x_{i,n-1}-x_{i,j}x_{j,n-1}\qquad & i<j\in J\\
&x_{j}x_{i,n}-x_{i,j}x_{j,n}\qquad & x_{j,n-1}x_{i,n}-x_{i,n-1}x_{j,n}\qquad &i<j\in J\\
&x_{i,n-1}x_{j,n}-x_{1,n}^2x_{i,j}\qquad & \quad &i<j\in J\\
&x_ix_{1}-x_{i,n-1}^2\qquad & x_ix_{n}-x_{i,n}^2\qquad &i\in J\\
&x_{i,n-1}x_{i,n}-x_{1,n}^2x_i\qquad & x_{i,n}x_{1}-x_{1,n}^2x_{i,n-1}\qquad & i\in J\\
&y_{i,1}-x_{1,n}x_{i,1}\qquad &y_{i,n}-x_{1,n}x_{i,n} \qquad & i\in J\\
&x_{i,n-1}x_{n}-x_{1,n}^2x_{i,n}\qquad &x_{1}x_{n}-x_{1,n}^4 \qquad & i\in J\\
\end{array}$$ is a Gr\"obner basis of
$I_{\mathcal{D}_{2m+1}}$ with respect to the ordering $\succ^{odd}$.
\end{thm}

\begin{proof} Let $M$ and $M'$ be two monomials in $K[D_{2m+1}]$ with $M \notin in(\mathcal{G}_{D_{2m+1}})$
and $M' \notin in(\mathcal{G}_{D_{2m+1}})$, where
$in(\G_{D_{2m+1}})$ is the monomial ideal generated by initial
terms of binomials in $\G_{D_{2m+1}}$. Since $y_{i,1}, y_{i,n},
x_ix_j \in in(\G_{D_{2m+1}})$, we may assume that
\begin{eqnarray*}
M &=&
x_{a}^px_{1}^{\alpha_{1}}x_{n}^{\alpha_n}x_{1,n}^{\beta}x_{b_1,c_1}\cdots
x_{b_q,c_q} x_{d_1,n-1}\cdots x_{d_r,n-1} x_{e_1,n}\cdots
x_{e_s,n} \quad \mbox{and} \quad\\
M'&=&x_{a'}^{p'}x_{1}^{\alpha'_{1}}x_{n}^{\alpha'_n}x_{1,n}^{\beta'}x_{b'_1,c'_1}\cdots
x_{b'_{q'},c'_{q'}} x_{d'_1,n-1}\cdots x_{d'_{r'},n-1}
x_{e'_1,n}\cdots x_{e'_{s'},n},
\end{eqnarray*} where the variables are ordered
with respect to
$$x_{a}\succ x_{b_1,c_1} \succ \cdots \succ x_{b_q,c_q}\succ x_{d_1,n-1}
\succ \cdots \succ x_{d_r,n-1}\succ x_{e_1,n} \succ \cdots \succ
x_{e_s,n},$$
$$x_{a'}\succ x_{b'_1,c'_1} \succ \cdots \succ x_{b'_{q'},c'_{q'}}\succ
x_{d'_1,n-1} \succ \cdots \succ x_{d'_{r'},n-1}\succ x_{e'_1,n}
\succ \cdots \succ x_{e'_{s'},n}.$$ First, we observe that the ordering above implies that $c_1 \leq
\dots \leq c_q$, $c'_1 \leq \dots \leq c'_{q'}$, $d_1 \leq
\dots \leq d_r$, $d'_1 \leq \dots \leq d'_{r'}$ and $e_1 \leq
\dots \leq e_r$, $e'_1 \leq \dots \leq e'_{r'}$. Moreover, we have
$b_1<c_1 \leq b_2<c_2 \leq \dots \leq b_q<c_q$  and $b'_1<c'_1
\leq b'_2<c'_2 \leq \dots \leq b'_{q'}<c'_{q'}$, since
$x_{i,k}x_{j,\ell}, x_{i,\ell}x_{j,k}, x_{i,j}x_{i,k} \in
in(\G_{D_{2m+1}})$.

The images of $M$ and $M'$ are
found as follows
\begin{eqnarray*}
\pi(M)&=&u_{a}^{2p}u_{1}^{4\alpha_{1}+\beta+2r}u_{n}^{4\alpha_n+\beta+2s}
u_{b_1}u_{c_1}\cdots u_{b_q}u_{c_q} u_{d_1}\cdots u_{d_r}u_{e_1}\cdots u_{e_s}\\
\pi(M')&=&u_{a'}^{2p'}u_{1}^{4\alpha'_{1}+\beta'+2r'}u_{n}^{4\alpha'_n+\beta'+2s'}
u_{b'_1}u_{c'_1}\cdots u_{b'_{q'}}u_{c'_{q'}} u_{d'_1}\cdots
u_{d'_{r'}}u_{e'_1}\cdots u_{e'_{s'}}.
\end{eqnarray*}

It follows from $\pi(M)=\pi(M')$ that we have
the following identities
\begin{eqnarray}
 \label{e1} 2p+2q+r+s &=& 2p'+2q'+r'+s' \\
  \label{e2} 4\alpha_{1}+\beta +2r &=& 4\alpha'_{1}+\beta' +2r' \\
  \label{e3}  4\alpha_{n}+\beta +2s &=& 4\alpha'_{n}+\beta' +2s'\\
   \label{e4} 2\alpha_{1}-2\alpha_{n}+r-s &=& 2\alpha'_{1}-2\alpha'_{n}+r'-s'  \: (\mbox{follows from} \: (\ref{e2}) \: \mbox{and} \: (\ref{e3})).
\end{eqnarray}

To accomplish our goal $M=M'$, we will assume contrarily that $M
\neq M'$ and obtain a contradiction in all possible cases
considered below. Since $M\neq M'$, we may suppose further that
they have no variable in common without loss of generality.

Since $x_{1}x_n \in in(\G_{D_{2m+1}})$, it follows that
$\alpha_{1}$ and $\alpha_n$ can not be positive simultaneously. If
$\alpha_{1}>0$ then $\alpha_n =0$ and $\alpha'_{1}=0$ immediately.
That $p=q=s=0$ follows respectively from $x_ix_{1}, x_{i,j}x_{1},
x_{i,n}x_{1}\in in(\G_{D_{2m+1}})$. Thus equations \ref{e1} and
\ref{e4} become
\begin{eqnarray*}
  r &=& 2p'+2q'+r'+s' \\
  2\alpha_{1}+r &=& -2\alpha'_{n}+r'-s'
\end{eqnarray*}

\noindent and we have $2\alpha_{1}=-2(p'+q'+s'+\alpha'_{n})\leq 0$,
contradiction. If $\alpha_{n}>0$ then clearly $\alpha'_n =0$ and
$\alpha_{1}=0$. That $p=q=r=0$ follows respectively from
$x_ix_{n}, x_{i,j}x_{n}, x_{i,n-1}x_{n}\in in(\G_{D_{2m+1}})$.
Thus equations \ref{e1} and \ref{e4} become
\begin{eqnarray*}
  s &=& 2p'+2q'+r'+s' \\
 -2\alpha_{n}-s &=& 2\alpha'_{1}+r'-s'
\end{eqnarray*}

\noindent and we have $2\alpha_{n}=-2(p'+q'+r'+\alpha'_{1})\leq
0$, contradiction. So, both $\alpha_1=\alpha_{n}=0$. One can show
that $\alpha'_{1}=0$ and $\alpha'_{n}=0$ by a similar argument.

Now, $x_{j,n-1}x_{i,n}, x_{i,n-1}x_{j,n}, x_{i,n-1}x_{i,n} \in
in(\G_{D_{2m+1}})$ implies that $r$ and $s$ (resp. $r'$ and $s'$) can not be positive at the same time.

If $r>0$, then $s=0$ in which case equation \ref{e4} becomes $r=r'-s'$. If $r'>0$, then $s'=0$ and we have $r=r'>0$, which is impossible as in this case,
$d_r$ would be equal to $d'_{r'}$ since these are the biggest indices of variables in $M$ and $M'$, $x_{d_r}$ would be a common variable.
If $s'>0$, then $r'=0$ and we have $r=-s'$, contradiction as $r>0$ and $s'>0$.

If $s>0$, then $r=0$ in which case equation \ref{e4} becomes
$-s=r'-s'$. If $r'>0$, then $s'=0$ and we have $-s=r'$, which
contradicts the assumption that $s>0$ and $r'>0$. If $s'>0$, then
$r'=0$ and we have $s=s'>0$, which is impossible as in this case
$e_s$ would be $e'_{s'}$ and since these are the biggest indices
of variables in $M$ and $M'$, $x_{e_s}$ would be a common
variable.

Hence, $r=s=0$ and this implies together with equation \ref{e4}
that $r'=s'$. Since they can not be positive simultaneously,
$r'=s'=0$ as well. After all these observations, equation \ref{e2}
reveals that $\beta=\beta'$. Since $M$ and $M'$ have no common
variable, it follows that $\beta=\beta'=0$.

If $q=0$ then $\pi(M)=\pi(M')$
implies that $u_{a}^{2p}= u_{a'}^{2p'}u_{b'_1}u_{c'_1}\cdots
u_{b'_{q'}}u_{c'_{q'}}$, which is possible only if $q'=0$ as
$b'_{q'}<c'_{q'}$. But in this case $a=a'$ and $x_{a}$ is a
common variable of $M$ and $M'$, a contradiction. Similarly, $q'=0$ gives rise to a contradiction. Thus $q>0$ and
$q'>0$.

Since $x_j x_{i,k}, x_k x_{i,j} \in in(\G_{D_{2m+1}})$, we have
$a \leq c_q$ and $a' \leq c'_{q'}$.  Since $b_q < c_q$ and
$b'_{q'} < c'_{q'}$, we observe that $c_q$ (resp. $c'_{q'}$) is
the biggest index appearing in $\pi(M)$ (resp. $\pi(M')$) which
yields together with $\pi(M)=\pi(M')$ that $c_q = c'_{q'}$. In
this case $u_{b_q}$ and $u_{b'_{q'}}$ appear in $\pi(M)=\pi(M')$.
Clearly $b_q>b'_{q'}$ or $b_q<b'_{q'}$, as otherwise $M$ and $M'$
would have a common variable $x_{b_{q},c_{q}}$. If
$b_q>b'_{q'}(>\dots>b'_1)$ then $b_q=a'$ as $u_{b_q}$ appears in
$\pi(M')$. This forces that $b'_{q'}<b_q=a'<c_q=c'_{q'}$ which
is impossible, since $x_jx_{i,k} \in in(\G_{D_{2m+1}})$. The other
case $b_q<b'_{q'}$ is impossible by a similar argument.\end{proof}

\begin{rem}
Note that if $n=2m+1$ we have,
\begin{center}
$\dis |\mathcal{G}_{D_{2m+1}}|=2\binom{m}{4}+7
\binom{m}{3}+9\binom{m}{2}+7 \binom{m}{1}+\binom{m}{0}.$
\end{center}

$\dim V_{D_{2m+1}}=2m+1$, $\codim V_{D_{2m+1}}=2m+1+\binom{m}{2}$.

\end{rem}

\subsection{$E_n$-type Singularities} We will give Gr\"obner
bases of toric ideals $I_{\mathcal{E}_n}$, where $n=6,7,8$, without proofs,
as they can easily be checked by a computation in Cocoa
\cite{cocoa}. To begin with, let us define the set $\mathcal{E}_6 \subset
\Z^6$:
$$\{3\e_1,3\e_2,\e_3,3\e_4,3\e_5,\e_6,\e_1+\e_2,\e_1+\e_5,\e_2+\e_4,\e_4+\e_5,2\e_2+\e_5,\e_2+2\e_5,2\e_1+\e_4,\e_1+2\e_4\}.$$
Let $K[\mathcal{E}_6]$ be the polynomial ring $K[x_1,\dots,x_{14}]$ with $14$ variables and $K[\N \mathcal{E}_6]$ be the semigroup ring generated
over $K$ by monomials $u^{\a}$ with $\a \in \mathcal{E}_6$. Then, as before, the toric ideal $I_{\mathcal{E}_6}$ is the kernel of the epimorphism defined by sending the $i$-th variable $x_i$ to $u^{\a_i}$, where $\a_i$ denotes the $i$-th element in $\mathcal{E}_6$, for all $i=1,\dots,14$. Similarly, we define the set $\mathcal{E}_7 \subset \Z^7$:
$$\{\e_1,\e_2,\e_3,2\e_4,\e_5,2\e_6,2\e_7,\e_4+\e_6,\e_4+\e_7,\e_6+\e_7\}.$$
Again, $K[\mathcal{E}_7]$ denotes the polynomial ring $K[x_1,\dots,x_{10}]$ with $10$ variables and $K[\N \mathcal{E}_7]$ be the semigroup ring generated
over $K$ by monomials $u^{\a}$ with $\a \in \mathcal{E}_7$. Thus, the toric ideal $I_{\mathcal{E}_7}$ is the kernel of the epimorphism defined by sending the $i$-th variable $x_i$ to $u^{\a_i}$, where $\a_i$ denotes the $i$-th element in $\mathcal{E}_7$, for all $i=1,\dots,10$. Finally, the set $\mathcal{E}_8 \subset \Z^8$ is defined as $\{\e_1,\dots,\e_8\}$.

\begin{thm} \label{En}  With the notations above we have the following:
\begin{enumerate}
\item A Gr\"obner basis for $I_{\mathcal{E}_6}$ with respect to lexicographic ordering with $x_1>x_2>x_3>x_4>x_5>x_6>x_{11}>x_{12}>x_{13}>x_{14}>x_7>x_8>x_9>x_{10}$ is given by
$$
\begin{array}{lllll}

&x_7x_{10}-x_8x_9, 				\quad& x_{13}x_{10}-x_{14}x_8,			\quad& x_{13}x_9-x_{14}x_7,			\quad& x_{12}x_{14}-x_8x_9x_{10},\\

&x_{12}x_{13}-x_8^2x_9, 	\quad& x_{11}x_{10}-x_{12}x_9,			\quad& x_{11}x_8-x_{12}x_7,			\quad& x_{11}x_{14}-x_8x_9^2,\\

&x_{11}x_{13}-x_7x_8x_9,  \quad& x_5x_9-x_{12}x_{10}, 				\quad& x_5x_7-x_{12}x_8, 				\quad& x_5x_{14}-x_8x_{10}^2,\\

&x_5x_{13}-x_8^2x_{10},   \quad& x_5x_{11}-x_{12}^2,					\quad& x_4x_8-x_{14}x_{10},			\quad& x_4x_7-x_{14}x_9,\\

&x_4x_{13}-x_{14}^2,	    \quad& x_4x_{12}-x_9x_{10}^2,				\quad& x_4x_{11}-x_9^2x_{10},		\quad& x_4x_5-x_{10}^3,\\

&x_2x_{10}-x_{11}x_9, 		\quad& x_2x_8-x_{11}x_7,						\quad& x_2x_{14}-x_7x_9^2,			\quad& x_2x_{13}-x_7^2x_9,\\

&x_2x_{12}-x_{11}^2,			\quad& x_2x_5-x_{11}x_{12},					\quad& x_2x_4-x_9^3,						\quad& x_1x_{10}-x_{13}x_8,\\

&x_1x_9-x_{13}x_7,				\quad& x_1x_{14}-x_{13}^3,					\quad& x_1x_{11}-x_7x_8^2,			\quad& x_1x_{11}-x_7^2x_8,\\

&x_1x_5-x_8^3,						\quad& x_1x_4-x_{13}x_{14},					\quad& x_1x_2-x_7^3 .

\end{array}$$

 \item  A Gr\"obner basis for $I_{\mathcal{E}_7}$ with respect to lexicographic ordering with $x_1>x_2>x_3>x_4>x_5>x_6>x_7>x_8>x_9>x_{10}$ is given by the following binomials
$$x_7x_8-x_9x_{10}, \quad x_6x_9-x_8x_{10}, \quad
x_6x_7-x_{10}^2,\quad x_4x_{10}-x_8x_9,\quad x_4x_7-x_9^2,\quad
x_4x_6-x_8^2.$$

\item The toric ideal $I_{\mathcal{E}_8}=(0).$
\end{enumerate}

\end{thm}

\section{Minimal generating sets} \label{sec04}

In this part, using ~\cite{charalambous-katsabekis-thoma} we show that the Gr\"obner bases obtained in the previous section are in fact minimal generating sets for each toric ideal. This will be achieved as follows.

Since our semigroups $\N\A$ are pointed, there is a partial order on them given by 

\begin{center}
$\mathbf{c} \leq \mathbf{d} \Leftrightarrow$ there is a
$\mathbf{c'} \in \N\A$ such that $\mathbf{c} +
\mathbf{c'}=\mathbf{d}$.
\end{center}

As $I_{\A}$ is  generated by binomials $x_\mathbf{a}-x_\mathbf{b}$ with $\pi(x_\mathbf{a})=\pi(x_\mathbf{b})$, $x_\mathbf{a}$ and $x_\mathbf{b}$ will have the same $\A$-\textit{degree}. Recall that for $\mathbf{p}=(p_1,\ldots,p_N)\in \N^N$, the $\A$-\textit{degree} of the monomial $x^\mathbf{p} := x_1^{p_1}\dots x_N^{p_N}$ is $\deg_{\A}(x^\mathbf{p})= p_1\mathbf{a}_1+\dots+p_N\mathbf{a}_N\in \N \A$. A vector $\mathbf{b}\in \N\A$ is called a \textit{Betti} $\A$-degree, if $I_{\A}$ has a minimal generating set containing an element of $\A$-degree $\mathbf{b}$. Since \textit{Betti} $\A$-degrees are independent of the minimal generating sets our Gr\"obner bases will determine all the candidate vectors $\mathbf{b}\in \N\A$.

For a vector $\textbf{b} \in \N\A$, $G(\textbf{b})$ is the graph
with vertices the elements of the fiber
$$deg_{\A}^{-1}(\textbf{b}) = \{x^\mathbf{p} \; | \; deg_{\A}(x^\mathbf{p})= \mathbf{b} \}$$
\noindent and edges all the sets $\{x^\mathbf{p}, x^\mathbf{q}\}$ , whenever
$x^\mathbf{p}-x^\mathbf{q} \in I_{\A,\mathbf{b}}$, where the ideal
$I_{\A,\mathbf{b}}$ is defined by $I_{\A,\mathbf{b}} = \langle
x^\mathbf{p}-x^\mathbf{p} \; |\; deg_{\A}(x^\mathbf{p})= deg_{\A}(x^\mathbf{q}) \lneqq
\textbf{b} \rangle$.

For each possible \textit{Betti} $\A$-degree $\textbf{b}$, we consider the complete graph $S_{\textbf{b}}$ with vertices
$G(\textbf{b})_{i}$, the connected components of $G(\textbf{b})$.
Let $T_{\textbf{b}}$ be a spanning tree of $S_{\textbf{b}}$. Then
$\mathcal{F}_{T_{\textbf{b}}}$ is the collection of binomials
$x^\mathbf{p}-x^\mathbf{q}$ corresponding to edges $\{x^\mathbf{p}, x^\mathbf{q}\}$ of
$T_{\textbf{b}}$ with $x^\mathbf{p} \in G(\textbf{b})_{i}$ and $x^\mathbf{q} \in
G(\textbf{b})_{j}$. We will use the following to show the
minimality of the generating sets given by the Gr\"obner bases presented in section 3.

\begin{thm}\cite[Theorem 2.6]{charalambous-katsabekis-thoma}.
$\mathcal{F}=\bigcup_{\mathbf{b} \in \N\A}
\mathcal{F}_{T_{\mathbf{b}}}$ is a minimal generating set of
$I_{\A}$.
\end{thm}
Notice that if $\mathbf{b}$ is not a $\textit{Betti}\, \A$-degree,
then $\mathcal{F}_{T_{\mathbf{b}}}=\emptyset$ and that the number of possible spanning trees determine the number of different minimal generating sets.

\subsection{Even Case  $\mathcal{D}_{2m}$} We consider the subset
$\mathcal{D}_{2m}$ defined by,
$$\dis \mathcal{D}_{2m}:=\{2\e_i, \e_j, 2\e_{1},2\e_{n}, \e_k+\e_{\ell},\e_i+\e_{1}+\e_n
\,|\, i,k, \ell \in J, \, j \in J^c \,\, \mbox{and} \,\, k<\ell \},$$ where
$J=\{3,5,\dots,n-1\}$ , $J^c=\{2,4,\dots,n-2\}$ and $\{\e_1,\dots,\e_n\}$ is the canonical basis of $\Z^n$ .
Recall that  the elements of $\mathcal{D}_{2m}$ are the $\mathcal{D}_{2m}$-degrees of the variables $x_i,\, x_j,\, x_1,\, x_n,\,x_{k,l}$ and $y_i$ respectively.

By Theorem 3.1, we see that $I_{\mathcal{D}_{2m}}$ is generated by the set $\mathcal{G}_{\mathcal{D}_{2m}}$\\[-3mm]
$$\begin{array}{lllll}
&x_{i,k}x_{j,\ell}-x_{i,j}x_{k,\ell} \quad &x_{i,\ell}x_{j,k}-x_{i,j}x_{k,\ell} \quad &i<j<k<\ell \in J\\
&x_{i,j}x_{i,k}-x_ix_{j,k} \quad &x_jx_{i,k}-x_{i,j}x_{j,k} \quad &i<j<k \in J \\
&x_kx_{i,j}-x_{i,k}x_{j,k} \quad & x_{j,k}y_i-x_{i,j}y_k \quad & i<j<k \in J \\
&x_{i,k}y_j-x_{i,j}y_k \quad & \quad & i<j<k \in J \\
&x_ix_j-x_{i,j}^2 \quad &x_jy_i-x_{i,j}y_j & i<j \in J\\
&x_{i,j}y_i-x_iy_j \quad &x_{i,j}x_{1}x_{n}-y_iy_j \quad & i<j \in J\\
&x_ix_{1}x_n-y_i^2 \quad & \quad & i \in J.
\end{array}$$

\noindent Therefore, possible \textit{Betti}
$\mathcal{D}_{2m}$-degrees are
$$\begin{array}{lllll}
& \mathbf{b_1}=2\e_i+2\e_j, \quad & \mathbf{b_2}=\e_i+\e_j+2\e_{1}+2\e_{n},\\
& \mathbf{b_3}=2\e_i+\e_j+\e_{1}+\e_{n},\quad & \mathbf{b_4}=\e_i+2\e_j+\e_{1}+\e_{n},\\
& \mathbf{b_5}=2\e_i+\e_j+\e_k,\quad & \mathbf{b_6}=\e_i+2\e_j+\e_k,\\
& \mathbf{b_7}=\e_i+\e_j+2\e_k,\quad & \mathbf{b_8}=\e_i+\e_j+\e_k+\e_{1}+\e_{n},\\
& \mathbf{b_9}=2\e_i+2\e_{1}+2\e_{n},\quad &
\mathbf{b_{10}}=\e_i+\e_j+\e_k+\e_{\ell}
\end{array}$$
Next we prove that these binomials constitute a minimal generating
set for $I_{ \mathcal{D}_{2m}}$.
\begin{prop}
The set $\mathcal{G}_{\mathcal{D}_{2m}}$ is a minimal generating set of
$I_{\mathcal{D}_{2m}}$.
\end{prop}
\begin{proof}
Since there is no binomial in $I_{\mathcal{D}_{2m},\mathbf{b_1}}$,
$G(\mathbf{b_1})$ consists of two connected components
$\{x_ix_j\}$ and $\{x_{i,j}^2\}$. Similarly, $G(\mathbf{b_2})$ has
$\{x_{i,j}x_{1}x_{n}\}$ and $\{y_iy_j\}$, $G(\mathbf{b_3})$ has
$\{x_{i,j}y_i\}$ and $\{x_iy_j\}$, $G(\mathbf{b_4})$ has
$\{x_jy_i\}$ and $\{x_{i,j}y_j\}$, $G(\mathbf{b_5})$ has
$\{x_{i,j}x_{i,k}\}$ and $\{x_ix_{j,k}\}$, $G(\mathbf{b_6})$ has
$\{x_jx_{i,k}\}$ and $\{x_{i,j}x_{j,k}\}$, $G(\mathbf{b_7})$ has
$\{x_kx_{i,j}\}$ and $\{x_{i,k}x_{j,k}\}$, $G(\mathbf{b_9})$ has
$\{x_ix_{1}x_n\}$ and $\{y_i^2\}$ as its connected components.

By Corollary $2.10$ in \cite{charalambous-katsabekis-thoma}, these
graphs determine all indispensable binomials of $I_{\mathcal{D}_{2m}}$.
Since these binomials are indispensable, they must belong to any
minimal generating set. Let us find the other binomials needed to
obtain a minimal generating set for $I_{\mathcal{D}_{2m}}$.

$G(\mathbf{b_8})$ and $G(\mathbf{b_{10}})$ have three connected
components:
$\{x_{i,j}y_k\}\cup\{x_{j,k}y_i\}\cup\{x_{i,k}y_j\}$ and
$\{x_{i,j}x_{k,\ell}\}\cup\{x_{i,k}x_{j,\ell}\}\cup\{x_{i,\ell}x_{j,k}\}$,
respectively. Since each connected component of these graphs is a
singleton, the complete graphs $S_{\mathbf{b_{8}}}$ and
$S_{\mathbf{b_{10}}}$  are triangles obtained by joining connected
components of $G(\mathbf{b_{8}})$ and $G(\mathbf{b_{10}})$,
respectively. Thus, spanning trees of these complete graphs can be
obtained by deleting one edge from the triangle.

Therefore, in a minimal generating set only one of the following
three binomial couples may appear corresponding to
$G(\mathbf{b_{8}})$;
\begin{center}
$x_{i,j}y_k-x_{j,k}y_i$ and $x_{i,j}y_k-x_{i,k}y_j$, or\\
$x_{j,k}y_i-x_{i,j}y_k$ and $x_{j,k}y_i-x_{i,k}y_j$, or\\
$x_{i,k}y_j-x_{i,j}y_k$ and $x_{i,k}y_j-x_{j,k}y_i$
\end{center}
and similarly for $G(\mathbf{b_{10}})$;

\begin{center}
$x_{i,j}x_{k,\ell}-x_{i,k}x_{j,\ell}$ and $x_{i,j}x_{k,\ell}-x_{i,\ell}x_{j,k}$, or\\
$x_{i,k}x_{j,\ell}-x_{i,j}x_{k,\ell}$ and $x_{i,k}x_{j,\ell}-x_{i,\ell}x_{j,k}$, or\\
$x_{i,\ell}x_{j,k}-x_{i,j}x_{k,\ell}$ and $x_{i,\ell}x_{j,k}-x_{i,k}x_{j,\ell}$.
\end{center}
\noindent Hence, there are many different minimal generating sets for the toric ideal $I_{\mathcal{D}_{2m}}$, and in particular the set $\mathcal{G}_{\mathcal{D}_{2m}}$ is a minimal generating set of $I_{\mathcal{D}_{2m}}$.
\end{proof}

\subsection{Odd Case  $\mathcal{D}_{2m+1}$} In this case, we
consider the set $\mathcal{D}_{2m+1} \subset \Z^{n}$ given by
\begin{center}
$\dis \{2\e_i,\e_j,4\e_{1},4\e_{n}, \e_k+\e_{\ell},\e_{1}+\e_n,
\e_i+2\e_{1},\e_i+2\e_{n}, \e_i+3\e_{1}+\e_n, \e_i+\e_{1}+3\e_n$ \\
\vspace{.3cm}$|\, i,k,\ell \in J, \,j\in J^c \,\, \mbox{and} \,\, k<\ell \},$
\end{center}

\noindent where  $J=\{2,4,\dots,n-1\}$, $J^c=\{3,5,\dots,n-2\}$ and $\{\e_1,\dots,\e_n\}$ is the canonical basis of $\Z^n$.
Again, the $\mathcal{D}_{2m+1}$-degrees of the variables are exactly the elements of $\mathcal{D}_{2m+1}$ as before.

By Theorem 3.3, we see that $I_{\mathcal{D}_{2m+1}}$ is generated by the set $\mathcal{G}_{\mathcal{D}_{2m+1}}$ of binomials

$$\begin{array}{lllll}
&x_{i,k}x_{j,\ell}-x_{i,j}x_{k,\ell} \qquad  &x_{i,\ell}x_{j,k}-x_{i,j}x_{k,\ell}  \qquad &i<j<k<\ell \in J\\
&x_{j,k}x_{i,n-1}-x_{i,j}x_{k,n-1}\qquad &x_{i,k}x_{j,n-1}-x_{i,j}x_{k,n-1}\qquad  &i<j<k\in J\\
&x_{j,k}x_{i,n}-x_{i,j}x_{k,n}\qquad &x_{i,k}x_{j,n}-x_{i,j}x_{k,n}\qquad & i<j<k\in J\\
&x_{j}x_{i,k}-x_{i,j}x_{j,k}\qquad &x_{i,j}x_{i,k}-x_{i}x_{j,k}\qquad & i<j<k\in J\\
&x_{k}x_{i,j}-x_{i,k}x_{j,k}\qquad & \quad & i<j<k\in J\\
&x_ix_j-x_{i,j}^2\qquad &x_{i,j}x_{1}-x_{i,n-1}x_{j,n-1}\qquad &i<j\in J\\
&x_{i,j}x_{n}-x_{i,n}x_{j,n}\qquad &x_{i,j}x_{i,n-1}-x_ix_{j,n-1}\qquad & i<j\in J\\
&x_{i,j}x_{i,n}-x_ix_{j,n}\qquad &x_{j}x_{i,n-1}-x_{i,j}x_{j,n-1}\qquad & i<j\in J\\
&x_{j}x_{i,n}-x_{i,j}x_{j,n}\qquad & x_{j,n-1}x_{i,n}-x_{i,n-1}x_{j,n}\qquad &i<j\in J\\
&x_{i,n-1}x_{j,n}-x_{1,n}^2x_{i,j}\qquad & \quad &i<j\in J\\
&x_ix_{1}-x_{i,n-1}^2\qquad & x_ix_{n}-x_{i,n}^2\qquad &i\in J\\
&x_{i,n-1}x_{i,n}-x_{1,n}^2x_i\qquad & x_{i,n}x_{1}-x_{1,n}^2x_{i,n-1}\qquad & i\in J\\
&y_{i,1}-x_{1,n}x_{i,1}\qquad &y_{i,n}-x_{1,n}x_{i,n} \qquad & i\in J\\
&x_{i,n-1}x_{n}-x_{1,n}^2x_{i,n}\qquad &x_{1}x_{n}-x_{1,n}^4 \qquad & i\in J\\
\end{array}$$
\noindent Therefore, possible Betti $\mathcal{D}_{2m+1}$-degrees
are
$$\begin{array}{lllll}
& \mathbf{b_1}=2\e_i+2\e_j,\qquad  & \mathbf{b_2}=4\e_{1}+\e_i+\e_j,\\
& \mathbf{b_3}=\e_i+\e_j+4\e_{n},\qquad& \mathbf{b_4}=2\e_{1}+2\e_i+\e_j,\\
& \mathbf{b_5}=2\e_i+\e_j+2\e_{n},\qquad& \mathbf{b_6}=2\e_{1}+\e_i+2\e_j,\\
& \mathbf{b_7}=\e_i+2\e_j+2\e_{n},\qquad& \mathbf{b_8}=2\e_{1}+\e_i+\e_j+2\e_{n},\\
& \mathbf{b_9}=2\e_{1}+\e_i+\e_j+\e_k ,\qquad & \mathbf{b_{10}}=\e_i+\e_j+\e_k+2\e_{n},\\
& \mathbf{b_{11}}=\e_i+2\e_j+\e_k,\qquad& \mathbf{b_{12}}=2\e_i+\e_j+\e_k,\\
& \mathbf{b_{13}}=\e_i+\e_j+2\e_k,\qquad& \mathbf{b_{14}}=\e_i+\e_j+\e_k+\e_{\ell},\\
& \mathbf{b_{15}}=4\e_{1}+2\e_i,\qquad& \mathbf{b_{16}}=2\e_i+4\e_{n},\\
& \mathbf{b_{17}}=2\e_{1}+2\e_i+2\e_{n},\qquad& \mathbf{b_{18}}=4\e_{1}+\e_i+2\e_{n},\\
& \mathbf{b_{19}}=3\e_{1}+\e_i+\e_{n},\qquad& \mathbf{b_{20}}=\e_{1}+\e_{i}+3\e_{n}\\
& \mathbf{b_{21}}=2\e_{1}+\e_i+4\e_{n},\qquad& \mathbf{b_{22}}=4\e_{1}+4\e_{n}.
\end{array}$$
Next we prove that these binomials constitute a minimal generating
set for $I_{\mathcal{D}_{2m+1}}$.

\begin{prop} The set $\mathcal{G}_{\mathcal{D}_{2m+1}}$ is a minimal generating set of $I_{\mathcal{D}_{2m+1}}$.
\end{prop}

\begin{proof} There is no binomial in $I_{\mathcal{D}_{2m+1},\mathbf{b_1}}$. Thus, the graph
$G(\mathbf{b_1})$ consists of two connected components
$\{x_ix_j\}$ and $\{x_{i,j}^2\}$. Similarly, $G(\mathbf{b_2})$ has
$\{x_{i,j}x_{1}\}$ and $\{x_{i,n-1}x_{j,n-1}\}$, $G(\mathbf{b_3})$
has $\{x_{i,j}x_n\}$ and $\{x_{i,n}x_{j,n}\}$, $G(\mathbf{b_4})$
has $\{x_{i,j}x_{i,n-1}\}$ and $\{x_ix_{j,n-1}\}$,
$G(\mathbf{b_5})$ has $\{x_{i,j}x_{i,n}\}$ and $\{x_ix_{j,n}\}$,
$G(\mathbf{b_6})$ has $\{x_jx_{i,n-1}\}$ and
$\{x_{i,j}x_{j,n-1}\}$, $G(\mathbf{b_7})$ has $\{x_jx_{i,n}\}$ and
$\{x_{i,j}x_{j,n}\}$, $G(\mathbf{b_{11}})$ has
$\{x_jx_{i,k}\}$ and $\{x_{i,j}x_{j,k}\}$, $G(\mathbf{b_{12}})$
has $\{x_{i,j}x_{i,k}\}$ and $\{x_{i}x_{j,k}\}$,
$G(\mathbf{b_{13}})$ has $\{x_{k}x_{i,j}\}$ and
$\{x_{i,k}x_{j,k}\}$, $G(\mathbf{b_{15}})$ has $\{x_{i},x_{1}\}$
and $\{x_{i,n-1}^2\}$, $G(\mathbf{b_{16}})$ has $\{x_{i},x_{n}\}$
and $\{x_{i,n}^2\}$, $G(\mathbf{b_{17}})$ has
$\{x_{i,n-1}x_{i,n}\}$ and $\{x_{1,n}^2x_{i}\}$,
$G(\mathbf{b_{18}})$ has $\{x_{i,n}x_{1}\}$ and
$\{x_{1,n}^2x_{i,n-1}\}$, $G(\mathbf{b_{19}})$ has
$\{y_{i,1}\}$ and $\{x_{1,n}x_{i,1}\}$, $G(\mathbf{b_{20}})$ has
$\{y_{i,n}\}$ and $\{x_{1,n}x_{i,n}\}$, $G(\mathbf{b_{21}})$ has
$\{x_{i,n-1}x_{n}\}$ and $\{x_{1,n}^2x_{i,n}\}$, and finally
$G(\mathbf{b_{22}})$ has $\{x_{1}x_{n}\}$ and $\{x_{1,n}^4\}$
as its connected components.

Indispensable binomials of $I_{\mathcal{D}_{2m+1}}$ are all determined by these graphs by Corollary $2.10$ in \cite{charalambous-katsabekis-thoma} and hence, corresponding binomials belong to any minimal generating set. 

The other graphs $G(\mathbf{b_8})$, $G(\mathbf{b_{9}})$, $G(\mathbf{b_{10}})$ and $G(\mathbf{b_{14}})$ have three connected components:
$$\{x_{i,n-1}x_{j,n}\} \cup \{x_{j,n-1}x_{i,n}\} \cup \{x_{1,n}^2x_{i,j}\},
\{x_{j,k}x_{i,n-1}\} \cup \{x_{i,j}x_{k,n-1}\} \cup
\{x_{i,k}x_{j,n-1}\}, $$
$$\{x_{j,k}x_{i,n}\} \cup \{x_{i,j}x_{k,n}\} \cup \{x_{i,k}x_{j,n}\} \;\mathrm{and}\;
\{x_{i,k}x_{j,\ell}\} \cup \{x_{i,j}x_{k,\ell}\} \cup \{x_{i,\ell}x_{j,k}\}$$

\noindent respectively. Each connected
component of these graphs is a singleton. Therefore, the complete graphs are triangles obtained by joining the connected
components of the graphs $G(\mathbf{b_{8}})$, $G(\mathbf{b_{9}})$, $G(\mathbf{b_{10}})$ and $G(\mathbf{b_{14}})$, respectively. Thus, we obtain the spanning trees by deleting one edge from each
triangle.

Therefore, in a minimal generating set only one of the following
three binomial couples may appear corresponding to
$G(\mathbf{b_{8}})$;
\begin{center}
$x_{i,n-1}x_{j,n}-x_{1,n}^2x_{i,j}$ and $x_{i,n-1}x_{j,n}-x_{j,n-1}x_{i,n}$,\\
$x_{1,n}^2x_{i,j}-x_{i,n-1}x_{j,n}$ and $x_{1,n}^2x_{i,j}-x_{j,n-1}x_{i,n}$,\\
$x_{j,n-1}x_{i,n}-x_{i,n-1}x_{j,n}$ and $x_{j,n-1}x_{i,n}-x_{1,n}^2x_{i,j}$
\end{center}
and the same is true for the following couples corresponding to
$G(\mathbf{b_{9}})$;
\begin{center}
$x_{j,k}x_{i,n-1}-x_{i,j}x_{k,n-1}$ and $x_{j,k}x_{i,n-1}-x_{i,k}x_{j,n-1}$,\\
$x_{i,j}x_{k,n-1}-x_{j,k}x_{i,n-1}$ and $x_{i,j}x_{k,n-1}-x_{i,k}x_{j,n-1}$,\\
$x_{i,k}x_{j,n-1}-x_{j,k}x_{i,n-1}$ and
$x_{i,k}x_{j,n-1}-x_{i,j}x_{k,n-1}$
\end{center}
and similarly for $G(\mathbf{b_{10}})$;
\begin{center}
$x_{j,k}x_{i,n}-x_{i,j}x_{k,n}$ and $x_{j,k}x_{i,n}-x_{i,k}x_{j,n}$,\\
$x_{i,j}x_{k,n}-x_{j,k}x_{i,n}$ and $x_{i,j}x_{k,n}-x_{i,k}x_{j,n}$,\\
$x_{i,k}x_{j,n}-x_{j,k}x_{i,n}$ and
$x_{i,k}x_{j,n}-x_{i,j}x_{k,n}$
\end{center}
and for $G(\mathbf{b_{14}})$;
\begin{center}
$x_{i,k}x_{j,\ell}-x_{i,j}x_{k,\ell}$ and $x_{i,k}x_{j,\ell}-x_{i,\ell}x_{j,k}$,\\
$x_{i,j}x_{k,\ell}-x_{i,k}x_{j,\ell}$ and $x_{i,j}x_{k,\ell}-x_{i,\ell}x_{j,k}$,\\
$x_{i,\ell}x_{j,k}-x_{i,j}x_{k,\ell}$ and
$x_{i,\ell}x_{j,k}-x_{i,k}x_{j,\ell}$.
\end{center}

\noindent These discussions show that there are many minimal generating sets for 
$I_{\mathcal{D}_{2m+1}}$ and in particular, the set
$\mathcal{G}_{\mathcal{D}_{2m+1}}$ is a minimal generating set of
$I_{\mathcal{D}_{2m+1}}$.
\end{proof}

\subsection{$E_n$-type} In this case, it is easy to check that the Gr\"obner basis given in Theorem \ref{En} constitutes a minimal generating set for each $n=6,\, 7,\, 8$. Indeed, there is nothing to prove for the case of $n=8$, as the corresponding toric ideal is trivial. In the case of $n=7$, the corresponding toric ideal is generated minimally by the $6$ binomials given in Theorem \ref{En} $(2)$ as we explain now. Let $\b$ be the $\mathcal{E}_7$-degree of a binomial given in Theorem \ref{En} $(2)$. Since the graph $G(\b)$ has two connected components, the complete graph $S_{\b}$ (and its spanning tree $T_{\b}$) is a line segment and thus $\mathcal{F}_{T_{b}}$ is a singleton. As the connected components of $G(\b)$ are singletons, $\mathcal{F}_{T_{b}}$ \textit{must} consist of the binomial we have started with. This means that the binomial is \textit{indispensable}, i.e. appears in any minimal generating set. Therefore the toric ideal has a unique minimal generating set provided by Theorem \ref{En} $(2)$. 

As for the case of $n=6$, we have a generating set given in Theorem \ref{En} $(1)$ consisting of $35$ binomials. Let $\b=2\e_1+2\e_2+\e_4+\e_5$ which is the $\mathcal{E}_6$-degree of the binomial $x_{11}x_{13}-x_7x_8x_9$. The graph $G(\b)$ has two connected components $\{x_{11}x_{13}\}$ and $\{x_7x_8x_9, x_7^2x_{10}\}$. As before the complete graph $S_{\b}$ (and its spanning tree $T_{\b}$) is a line segment and thus $\mathcal{F}_{T_{b}}$ is a singleton but it changes according to which monomial we choose from the second component of $G(\b)$. So, $\mathcal{F}_{T_{b}}$ is either $\{x_{11}x_{13}-x_7x_8x_9\}$ or $\{x_{11}x_{13}-x_7^2x_{10}\}$. We have the same situation for the following degrees:
$$\begin{array}{lllll}
& \mathbf{b}=\e_1+2\e_2+2\e_4+\e_5,\\
& \mathbf{b}=2\e_1+\e_2+\e_4+2\e_5,\\
& \mathbf{b}=\e_1+\e_2+2\e_4+2\e_5.
\end{array}$$
\noindent It is a standard procedure to check that the other $31$ binomials given in Theorem \ref{En} $(1)$ are indispensable, so there are $8$ different minimal generating sets for the toric ideal including the one provided by Theorem \ref{En} $(1)$.

\section{What about $A_n$-type?}
There are two ways to study the question of whether or not toric ideals of these configurations have squarefree initial ideals. The first one is to produce an example with no squarefree initial ideal using computer programs. In order to achieve this goal one has to find all possible initial ideals for a fixed configuration. The toric ideal corresponding to $A_2$ is generated by a binomial with a squarefree monomial. One can compute $29$ different initial ideals for the toric ideal of $A_3$ and obtain the unique squarefree one generated by $6$ monomials by using e.g. Gfan \cite{Gfan}. As long as $n$ gets larger values listing all the possible initial ideals (or regular triangulations of the corresponding convex polytope) using computer programs becomes problematic. In the second way, one has to determine the correct term order with respect to which the initial ideal is generated by squarefree monomials by heuristic/experimental methods. For the toric ideal of $A_4$ the lexicographic ordering with $x_{14} >x_{12} >x_{10} >x_{9} >x_{7} >x_{4} >x_{8} >x_{6} >x_{5} >x_{3} >x_{11} >x_{1} >x_{2}$ gives a Gr\"obner basis consisting of $54$ binomials with a squarefree initial ideal. Similarly, the toric ideal of $A_5$ has a squarefree initial ideal generated by $105$ monomials which are obtained as the initial terms with respect to the lexicographic ordering with $x_{19} >x_{18} >x_{17} >x_{11} >x_{10} >x_{3} >x_{16} >x_{13} >x_{7} >x_{15} >x_{14} >x_{12} > x_{8} >x_{5} >x_{9} >x_{4} >x_{6} >x_{2} >x_{1}.$ However, for larger values of $n$, proving the existence of squarefree initial ideals is difficult as well. This is due to the fact that there is no general formula for the vector configuration as in the case of $D$-type, although one can compute them one by one with e.g. CoCoA using the algorithm described in \cite{mesut}.

\section*{Acknowledgments} The authors would like to thank M. Tosun for increasing their interest on toric varieties of rational surface singularities and A. Jensen for correspondences about the possibility to use Gfan to enumerate all initial ideals. They also thank the anonymous referee for useful suggestions which improved the presentation of the paper. 

\end{document}